\pdfoutput=1
\documentclass{amsart}
\usepackage{geometry}
\usepackage[pdftex]{graphicx}
\usepackage{hyperref}
\usepackage{mathrsfs}
\theoremstyle{plain}
\newtheorem{teo}{Theorem}
\newtheorem{pro}{Proposition}

\newtheorem{lem}{Lemma}
\newtheorem{cor}{Corollary}
\theoremstyle{definition}

\theoremstyle{remark}

\title[The Brownian height fragmentation]{The falling apart of the tagged fragment and the asymptotic disintegration of the Brownian height fragmentation}
\author{Ger\'onimo Uribe Bravo}
\address{Instituto de Matem\'aticas, U.N.A.M.\\
	\'Area de la investigaci\'on cient\ii fica\\
	Circuito Exterior	Ciudad Universitaria \\
	Coyoac\'an 04510\\
	M\'exico, D.F. M\'exico
}
\thanks{Research supported by CoNaCyT grant No. 174498}
\email{uribe@matem.unam.mx}
\date{\today}

\newcommand{\paren}[1]{\ensuremath{\left( #1\right) }}
\newcommand{\e}{\ensuremath{\mathbf{e}}}
\newenvironment{esn}{\begin{equation*}}{\end{equation*}}
\newcommand{\abs}[1]{\hspace{.25mm}\lvert#1\rvert\hspace{.25mm}}
\newcommand{\set}[1]{\ensuremath{\left\{ #1\right\} }}
\newcommand{\mc}[1]{\ensuremath{\mathscr{#1}}}
\newcommand{\fun}[3]{\ensuremath{#1:#2\to #3}}
\newcommand{\re}{\ensuremath{\mathbb{R}}}
\newcommand{\imf}[2]{\ensuremath{#1\!\paren{#2}}}
\newcommand{\eps}{\ensuremath{ \varepsilon}}
\newcommand{\indi}[1]{\si_{#1}}
\newcommand{\si}{{\ensuremath{\bf{1}}}}
\newcommand{\na}{\ensuremath{\mathbb{N}}}
\newcommand{\sfleche}{\mc{S}^{\downarrow}}
\newcommand{\esp}[1]{\ensuremath{\se\! \left( #1 \right)}}
\newcommand{\se}{\ensuremath{\bb{E}}}
\newcommand{\bb}[1]{\mathbb{#1}}
\newcommand{\p}{\ensuremath{ \sip  } }
\newcommand{\sip}{\bb{P}}
\DeclareMathOperator{\pd}{PD}
\newcommand{\probac}[2]{\ensuremath{\proba{\cond{#1}{#2}}}}
\newcommand{\cond}[2]{\left.\vphantom{#2}#1\ \right| #2}
\newcommand{\proba}[1]{\ensuremath{\sip\! \left( #1 \right)}}
\newcommand{\F}{\ssa}
\newcommand{\ssa}{\ensuremath{\mathscr{F}}}
\newcommand{\sag}[1]{\sigma\!\paren{#1}}
\newcommand{\z}{\ensuremath{\mathbb{Z}}}
\newcommand{\sa}{\ensuremath{\sigma}\nobreakdash-field}
\newcommand{\fund}[3]{\ensuremath{#1:#2\mapsto #3}}
\newcommand{\nbd}{\nobreakdash -}
\newcommand{\ii}{\'{\i}}

\begin{document} 
\maketitle
\begin{abstract}
We present a further analysis of the fragmentation at heights of the normalized Brownian excursion. Specifically we study a representation for the mass of a tagged fragment in terms of a Doob transformation of the $1/2$-stable subordinator and use it to study its jumps; this accounts for a description of how a typical fragment falls apart. These results carry over to the height fragmentation of the stable tree. Additionally, the sizes of the fragments in the Brownian height fragmentation when it is about to reduce to dust are described in a limit theorem. 
\end{abstract}
\renewcommand{\abstractname}{R\'esum\'e}
\begin{abstract}
Une \'etude additionnelle de la fragmentation de hauteur brownienne est pr\'esent\'ee. Plus pr\'ecis\'ement, une repr\'esentation de la masse du fragment marqu\'e en termes d'une transformation de Doob du subordinateur stable d'indice $1/2$ est d\'ecrite puis utilis\'ee pour \'etudier les sauts du processus de masse; ceci nous renseigne sur la fa\c con  dans laquelle un fragment typique se casse. Ces r\'esultats se g\'en\'eralisent au cadre des fragmentations de hauteur de l'arbre stable.  Enfin, nous donnons un th\'eor\`eme limite de la fragmentation de l'excursion Brownienne par 
les hauteurs, centr\'ee autour du dernier fragment qui se d\'ecompose en poussi\`ere.
\end{abstract}

\section{Introduction and statement of the results}
\label{intro}
 A new class of stochastic processes, that of self-similar fragmentations, has been  introduced by Bertoin in \cite{homfrag} and \cite{ssfrag} and is part of the subject of the book \cite{bertoinFragCoag}. Informally, a self-similar interval fragmentation is a model for the splitting of $(0,1)$ into smaller and smaller (open) pieces in such a way that the evolution is Markovian and that the evolution of the process is independent on different components and restricted  to each component, it mimics the whole fragmentation, though maybe on a different time scale which depends on a power of its size. Self -similar fragmentations are a parametric class of processes characterized by the self-similarity index, the rate of erosion, and the so-called L\'evy measure describing sudden dislocations. This class bears a close relationship with that of positive self-similar Markov processes, for which there has been renewed interest in recent years.  As Aldous  points out in his survey  \cite{mean-field}, fragmentation processes might be of use in the study of coalescence phenomena. This idea has been exemplified by Aldous and Pitman in \cite{sacoalescent} in the construction of the Standard Additive Coalescent by time-reversing a fragmentation process constructed from the Continuum Random Tree. They show that a tagged fragment of their fragmentation can be represented as the multiplicative inverse of a stable subordinator which starts at $1$; in their own words, the relationship  should be obtainable directly from the one between the CRT and the normalized Brownian excursion  (they use a combinatorial method). Bertoin has showed in \cite{ssfrag} that the Aldous-Pitman fragmentation and the height fragmentation of the normalized Brownian excursion (or equivalently, the height fragmentation of the CRT), which is the subject of this work, differ only by their self-similarity index so that tagged fragments of both fragmentations are related by a time-change. This remark is one of the motivations for the following paragraphs since we provide a  representation of the tagged fragment of the height fragmentation of the CRT in terms of the opposite of a stable subordinator conditioned to reach zero continuously. This is done in the framework of continuous-time stochastic processes. 
 
We shall study the height fragmentation of the CRT, which was introduced by Bertoin as the second example illustrating the  theory  of self-similar fragmentations developed in \cite{ssfrag}. We shall also deal with its generalization to height fragmentations of $\alpha$-stable trees (for $\alpha\in (1,2]$) introduced by Miermont in \cite{heights}.  Although both processes can be thought to belong to the same parametric family $\paren{F^\alpha}_{\alpha\in (1,2]}$ of fragmentations, their irreconcilable difference lies in the fact that, following Miermont,   the first one is binary while the second one is infinitary. This means that fragments separate into two pieces in the $\alpha=2$ case and into infinitely many pieces when $\alpha\in (1,2)$. This difference is the reflection of the fact that while Brownian trajectories have continuous sample paths, other stable L\'evy processes feature jumps and  affects the sophistication of the arguments needed to study them as can be seen in \cite{duquesnelegall,levytrees,heights}. As we hope to make apparent in this note, past an initial threshold, aspects of both fragmentations can be studied without recourse to different arguments.   However, the Brownian case trivially admits a representation in terms of the normalized Brownian excursion which makes a more visual  analysis feasible; the corresponding visual analysis for $\alpha\in(1,2)$ is more technical and would be based on the height process coding L\'evy trees first introduced in \cite{LeGallLeJan} by Le Jan and Le Gall and subsequently developed in \cite{duquesnelegall} and \cite{levytrees} by Duquesne and Le Gall. Since some of our results are valid for all $\alpha\in(1,2]$, we choose to present both the Brownian proof, which only suggests how L\'evy trees could be used, and the general proof that holds for all parameters and is conceived to use less of this technical machinery, when possible. However, note that some of the tools already available to study L\'evy trees are not directly applicable to our case since our fragmentations are built from stable trees conditioned by their size. 

We now turn to a more formal recollection of the processes mentioned above which will let us state our main results.

\subsection{The Brownian height fragmentation}
This is constructed, in \cite{ssfrag}, from the normalized Brownian excursion, which is the process $\e$ obtained from a Brownian motion $B$ by the following procedure: let $g_t$ be the last zero of $B$ before $t$ and $d_t$ the first zero of $B$ after some fixed time $t$ and define\begin{esn}
\e_s=\frac{1}{\sqrt{d_t-g_t}}\abs{B_{g_t+s\paren{d_t-g_t}}},\quad s\in [0,1].
\end{esn}
The law of $\e$ does not depend on $t$ by the scaling properties of $B$. The Brownian height fragmentation is defined as follows: for nonnegative $t$, let\begin{equation}
\label{browfragdef}
F^2_t=\set{s\in (0,1):\e_s>t};
\end{equation}then the Brownian height fragmentation is the decreasing family of sets given by $F^2=\paren{F^2_t}_{t\geq 0}$. In Figure \ref{fragpic}, a visualization of $F^2_t$ is proposed, with some other quantities of interest that shall be introduced in the following paragraphs. The process $F^2$ takes values in the space $\mc{V}$ of open subsets of $(0,1)$, where a suitable metric exists  which turns it into a compact space.%
\begin{figure}
\includegraphics{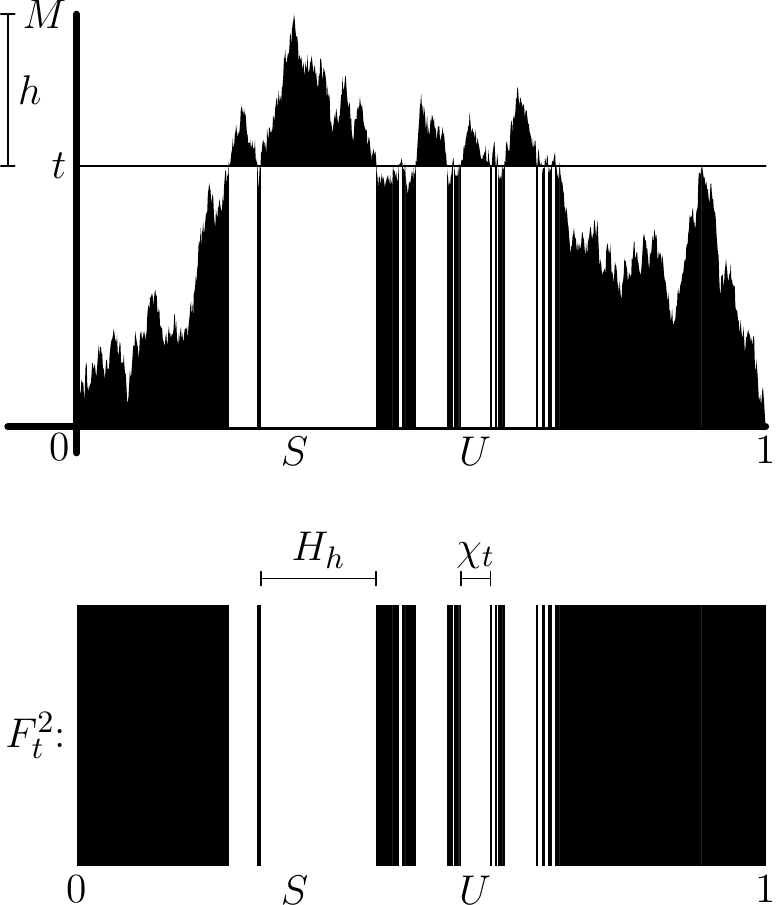}
\caption{Visualization of $F^2_t$}
\label{fragpic}
\end{figure}
\subsection{The tree interpretation}
As explained in Section 2 of  \cite{randomrealtrees}, given a nonnegative continuous function $\fun{f}{[0,1]}{\re_+}$ and such that $\imf{f}{0}=0$ (it will be referred to as the coding function,) a pointed metric space $\paren{\tau_f,d_f,\rho_f}$ belonging to the space of compact real trees can be constructed as follows:  define the pseudo-metric $d_f$ and the equivalence relation $\stackrel{f}{\sim}$ on $[0,1]$ by\begin{esn}
\imf{d_f}{s_1,s_2}=\imf{f}{s_1}+\imf{f}{s_2}-2\imf{m_f}{s_1,s_2},
\text{ where }\imf{m_f}{s_1,s_2}=\min_{r\in [s_1\wedge s_2,s_1\vee s_2]}\imf{f}{r}, 
\end{esn}and\begin{esn}
s_1\stackrel{f}{\sim} s_2\text{ if and  only if }\imf{d_f}{s_1,s_2}=0.
\end{esn}Then the quotient space $\tau_f=[0,1]/\stackrel{f}{\sim}$, with the induced distance (which will keep the notation $d_f$), is a compact real tree, to be rooted at  the equivalence class of $0$, denoted $\rho_f$. The locations of the local minima of $f$ code the nodes or branching points of $\tau_f$ (branching points disconnect the tree into more than two parts when removed). The equivalence classes of $s_1$ and $s_2$ ($s_1<s_2$) will branch from a common node, say $[s]$ satisfying $s_1< s< s_2$, exactly when $f$ is greater than $\imf{f}{s}$ on $[s_1,s_2]$. Finally, the height of an element of the tree is its $d_f$-distance to the root $\rho_f$ of the tree.  A visualization of the tree coded by a continuous function is given in Figure \ref{treeCoded}.
\begin{figure}
\includegraphics{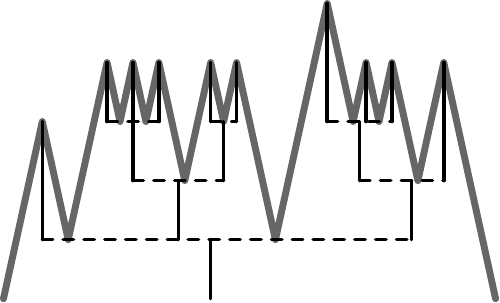}
\caption{Tree coded by a continuous function. (Join the slashed horizotal lines.)}
\label{treeCoded}
\end{figure}When the function $f$ is replaced by a random continuous function, such as $\e$, it gives rise to a random tree. Here, we see that the image under the canonical projection from $[0,1]$ to $[0,1]/\stackrel{\e}{\sim}$ gives a bijection between $F_t$ and the elements of $\tau_{\e}$ of height greater than $t$.

\subsection{The $\alpha$-stable height fragmentation}

A generalization of the Brownian height fragmentation to stable L\'evy processes of index $\alpha\in (0,2)$ is constructed from what has been termed the height process of the normalized $\alpha$-stable excursion for $\alpha\in (1,2]$. Firstly, the normalized $\alpha$-stable excursion was introduced by Chaumont in \cite{excursionNormaliseeStable}: given $\alpha\in (1,2]$, let $X$ be a spectrally positive L\'evy process of index $\alpha$ and define its cumulative infimum process $\underline X$ by\begin{esn}
\underline X_t=\inf_{s\leq t}X_s.
\end{esn}Define also $\underline g_t$ and $\underline d_t$ as the last (respectively first) instant before (respectively after) $t$ for which $X$ equals $\underline X$ and consider the normalized $\alpha$-stable excursion $\e^\alpha$ defined by\begin{esn}
\e^\alpha_s=\frac{1}{\paren{\underline d_t-\underline g_t}^{1/\alpha}}\paren{X-\underline X}_{\underline g_t+s\paren{\underline d_t-\underline g_t}}.
\end{esn}Since $\abs{B}$ has the same law as $B-\underline B$, the two constructions of the normalized Brownian excursion coincide. Second, Duquesne and Le Gall have defined in \cite{duquesnelegall} the height process $\mathbf{H}$ of $\e^\alpha$ as a continuous modification of\begin{esn}
u\mapsto \lim_{\eps\to 0}\int_0^u \indi{\overline\e^{u,\alpha}_s-\e^{u,\alpha}_s\leq \eps}\, ds
\end{esn}where $\e^{u,\alpha}$ is  given by $\e^{u,\alpha}_s=\e^\alpha_u-\e^\alpha_{\paren{u-s}-}$ for $s\leq u$ and $\overline \e^{u,\alpha}$ is the cumulative supremum process of $\e^{u,\alpha}$ given by\begin{esn}
\overline \e^{u,\alpha}_t=\sup_{s\leq t} \e^{u,\alpha}_s.
\end{esn}Then the $\alpha$-stable height fragmentation $F^\alpha$ is defined by\begin{esn}
F^\alpha_t=\set{s\in (0,1):\mathbf{H}_s>t}.
\end{esn}In the Brownian case, the height process has the law  of $2\e$, which implies the equivalence of the two definitions up to a multiplicative factor. We might also visualize this height fragmentation in terms of a fragmentation of the tree coded by $\mathbf{H}$, which is called the $\alpha$-stable tree (of size $1$), since this process is continuous. This self-similar fragmentation process has been introduced by Miermont in \cite{heights}. There are other ways of fragmenting this tree than by cutting down what is below a given height, Miermont introduced in \cite{hubs} the fragmentation at nodes which has been studied for other L\'evy trees by Abraham and Delmas in \cite{fragmentationUsingSnake}. 

In both cases, we can change the state space of the fragmentation from the space of open sets of $(0,1)$ to the space of partitions of $\na$, therefore providing a discrete framework for our investigations. This point of view was used by Bertoin in \cite{ssfrag} to construct and characterize self-similar fragmentations in terms of the self-similarity index, the erosion coefficient, and the dislocation measure analogously to the L\'evy-It\^o decomposition of L\'evy processes. In our case, thanks to \cite{ssfrag} and \cite{heights}, the self-similarity index is related to the index of the stable L\'evy process we are working with and is equal to $-1/2$ for the Brownian height fragmentation and equal to $1/\alpha-1\in (-1/2,0)$ for the $\alpha$-stable height fragmentation. The erosion coefficient is zero in both cases while the dislocation measure, defined on the set $\sfleche=\set{\mathbf{s}=\paren{s_1,s_2,\ldots}:s_1\geq s_2\geq \cdots\geq 0, \sum s_i\leq 1}$, charges only the set $\set{s_1>s_2>0,s_1+s_2=1}$ in the Brownian case and is characterized, in \cite[p. 340]{ssfrag}, by\begin{esn}
\imf{\nu_\e}{s_1\in dx}=\frac{2}{\sqrt{2\pi x^3\paren{1-x}^3}}\indi{x\in [1/2,1)}\,dx.
\end{esn}In the stable case, the dislocation measure is decribed, in \cite[p. 426]{heights}, in terms of a stable subordinator $\paren{T_t}_{t\in [0,1]}$ with Laplace exponent $\lambda\mapsto \lambda^{1/\alpha}$. Let $\Delta T_{[0,1]}$ denote the sequence of the jumps of $T$ on $[0,1]$ ranked in decreasing order; then the dislocation measure of the $\alpha$-stable height fragmentation is\begin{esn}
\imf{\nu_{-}}{d\mathbf{s}}=\alpha^2\frac{\imf{\Gamma}{2-1/\alpha}}{\imf{\Gamma}{2-\alpha}}\esp{T_1;\frac{\Delta T_{[0,1]}}{T_1}\in d\mathbf{s}}.
\end{esn}

\subsection{The results}
A simple process tied to any self-similar interval fragmentation is the mass of a tagged fragment. Instead of tracking down the behavior of the whole fragmentation process, we select the interval that contains an independent uniform random variable, with the objective of recording how its mass is lost. As Bertoin proves in \cite{ssfrag} passing to the space of partitions of $\na$ and using the results of \cite{homfrag}, the mass of the tagged fragment is a decreasing and  positive self-similar Markov process which can be therefore represented in terms of a L\'evy process using the  Lamperti transformation. 
Information about the erosion coefficient and the dislocation measure can sometimes be inferred from the study of a tagged fragment. A tagged fragment is defined as follows: let  $U$ be a uniform  random variable on $(0,1)$ independent of the fragmentation process $F^\alpha$, a tagged fragment at time $t$ is the component of $F^\alpha_t$ that contains $U$. Its size, will be denoted by $\chi_t$. Our first result gives us some information on how a tagged fragment falls appart; specifically, since it is proved in \cite{ssfrag} that the mass of a tagged fragment decreases only by jumps, we study their sizes. Let $\beta=1-1/\alpha \in (0,1/2]$.
\begin{teo}
\label{teosaltos}
The law of the decreasing rearrangement of the absolute values of the jumps of the mass of a tagged fragment of  $F^\alpha$ is the two-parameter Poisson-Dirichlet distribution with parameters $\paren{\beta,\beta}$.
\end{teo}
Further information regarding the two-parameter Poisson-Dirichlet distribution is found in the survey paper \cite{poissondirichlet}. Although the preceding theorem might be thought to be a consequence of general results on L\'evy trees (like the ancestral line decomposition of \cite{levytrees}), let us remark that the stable tree we are considering is conditioned by its size, while most of the results pertaining L\'evy trees work for the unconditioned measures.

The preceding theorem is analyzed in Section \ref{fallingapart}. In the Brownian case, it is explained by a visual argument relying on a path transformation between the normalized Brownian excursion and the Brownian bridge introduced in \cite{pathtrans}. Also, its connection to a fragmentation obtained by obliteration of ancestral lines in the CRT and the coagulation and fragmentation operators of Dong-Goldschmidt-Martin (cf. \cite{coagfrag}) is mentioned. This is used to give a Brownian construction of a self-similar fragmentation with zero erosion, self-similarity index one and whose dislocation measure is the Poisson-Dirichlet distribution of parameters $1/2$ and $1/2$. Regarding the general case, during the peer review process, Haas, Pitman and Winkel obtained Theorem \ref{teosaltos} in the discrete framework of partitions of $\na$ in \cite{haasPitmanWinkel}. We handle Theorem \ref{teosaltos} by the use of a formula by Perman, Pitman and Yor (see \cite[Formula (2.d), Theorem 2.1]{permanpitmanyor}) pertaining the law of the jumps of subordinators in conjunction with a description of the conditional law of a tagged fragment given its death time, which is done in this paper. The relevance of subordinators to the study of a tagged fragment does not come directly form the subordinators attached to a tagged fragment by the Lamperti transformation alluded to above, but from the following relationship between a tagged fragment and a stable subordinator of index $\beta$. 
\begin{pro}
\label{conditionedsubordinator}
The mass of the tagged fragment of $F^\alpha$ has the same law as the opposite of a stable subordinator of index $\beta$ starting at one and conditioned to die at zero. 
\end{pro}
 The stochastic process referred to in the statement  is a Doob transform of a stable subordinator via its potential density and is described as follows: let $\sigma$ be a stable subordinator of index $\beta$ (the construction of the conditioned subordinator allows for $\beta\in (0,1)$, but the reader should keep in mind that for us, $\beta=1-1/\alpha\in (0,1/2]$), with Laplace exponent $\psi$ given by $q\mapsto Cq^{\beta}$ for nonnegative $q$. It is known that the law of $\sigma_t$ admits a density $f_t$ for positive $t$, for which there is no simple explicit expression except in the case $\beta=1/2$, and  that the potential operator of $\sigma$ admits a density $\imf{u}{x,y}=\imf{u}{0,y-x}$ given explicitly by $\imf{u}{0,y}=\indi{y\geq 0}/C\imf{\Gamma}{\beta}y^{1-\beta}$. Since $\sigma$ and $-\sigma$ are in duality with respect to Lebesgue measure, it follows that the potential density of $-\sigma$ is $\imf{\hat{u}}{y,x}=\imf{u}{x,y}$. Also, $\imf{h}{x}=\imf{\hat{u}}{x,0}$ is a potential for the semigroup of $-\sigma$, since if $\set{\hat{P}_t:t\geq 0}$ denotes the semigroup of $-\sigma$ then\begin{equation}
\label{deathtimedistribution}
\imf{\hat{P}_t h}{x}=\int_t^\infty \imf{f_s}{x}\, ds\to 0\quad\paren{t\to\infty}.
\end{equation}So,  we might consider the Doob transformation of $-\sigma$ by $h$, $-\sigma^h$, and we shall denote its sub-Markovian family of distributions by $\set{\hat{\p}^h_x:x>0}$. Under $\hat{\p}^h_x$, the process dies almost surely in finite time because of the potential character of $h$: the left-hand side of \eqref{deathtimedistribution} divided by $\imf{h}{x}$ is the probability that, starting at $x$, the death-time $\zeta^h$ of $-\sigma^h$ is greater than $t$. The $h$-path process $-\sigma^h$ will be called the opposite of a stable subordinator of index $\beta$ starting at $x$ and conditioned to die at zero; the interpretation is justified because $-\sigma^h_{\zeta^h-}=0$ almost surely (see \cite[Sect. 4 Prop. 2]{chaumontpathdecomp}). This process was introduced, in a more general context, in \cite{chaumontpathdecomp}. There, the author obtains the behaviour near death-time, relying on the classification of coharmonic and coinvariant functions for L\'evy processes of \cite{silversteincoharmonic} and therefore, we shall present a way of obtaining it in Section \ref{representationtaggedfragment} by an approach closer to the techniques used in this paper, namely, the use of the Markovian bridges introduced in \cite{markovbridges}. The interested reader is referred to \cite{caballerochaumontsuborestable} for expressions of the infinitesimal generator and other aspects of the conditioned subordinator, but we will supply all the necessary tools to study this process.

We will prove Proposition \ref{conditionedsubordinator} by a visual analysis relying on It\^o measure of positive excursions in the Brownian case. The general case can be based on two different analyses pertaining positive self-similar Markov processes: the first one relies on the identification of the subordinator associated to a tagged fragment by Lamperti's representation, performed in the brownian case in \cite{ssfrag} and in the stable case in \cite{heights} as well as the duality considerations involving positive self-similar Markov processes of \cite{entrancelaws}. However, it is also shown how to bypass these identifications by use of the characterization of the death-times of the fragmentations performed by Duquesne and Le Gall (a consequence of \cite[Theorem 3.3.3]{duquesnelegall}) and the basic formula for the moments of exponential functionals of subordinators found in \cite[Formula (4), p. 194]{surveyexponentialfunctionals}.

Coming back to Theorem \ref{teosaltos}, since the Poisson-Dirichlet distribution with parameters $\paren{\beta,\beta}$ arises as the distribution of the ranked lengths of the excursions of a Bessel bridge of dimension $2\paren{1-\beta}$ starting and ending at zero (see \cite[(16), p. 860]{poissondirichlet} and the references therein), we shall link Theorem \ref{teosaltos} with Proposition \ref{conditionedsubordinator} in the following result which was suggested by them, but is actually independent and appears to be new. Together with  the analysis of Theorem \ref{teosaltos} and Proposition \ref{conditionedsubordinator}, it gives a different proof of the aformentioned result on the ranked lengths of excursions of Bessel bridges. 
\begin{pro}
\label{localTimeBesselBridge}
The inverse local time at zero of a $2\paren{1-\beta}$-dimensional Bessel bridge of length one starting and ending at zero is a stable subordinator of index $\beta$ starting at zero and conditioned to die at $1$. 
\end{pro}
Our last results concern limit theorems for the Brownian height fragmentation at the moment where it reduces to dust. To be more specific, let us note that, because of the continuity of $\e$, the first level $t$ at which $F^2_t=\emptyset$ exists, is finite and equal to the maximum of $\e$, denoted by $M$. Secondly, when we replace the deterministic level $t$ by the random one $M-t$ in \eqref{browfragdef}, we obtain a random variable with values in $\mc{V}$ which will be denoted $\hat{F}^2_t$.  Consider  two independent realizations $R$ and $R'$ of the Bessel process of dimension three starting at zero and set $Z_t=R_t$ if $t>0$ and $R'_{-t}$ if $t<0$. Finally, let $S\in (0,1)$ be the almost surely unique location of the maximum $M$, then the following holds:
\begin{teo}
\label{asymptotics}
As $t\to 0+$, the stochastic process\begin{esn}
\paren{\frac{\hat{F}^2_{rt}-S}{t^2}}_{r\geq 0}
\end{esn}with values in the set of open sets of $\re$ converges in distribution to \begin{esn}
\paren{\set{s\in\re:Z_s<r}}_{r\geq 0}.
\end{esn}
\end{teo}Apart from studying the proper topology on the family of open sets of $\re$ in Section \ref{extinction}, the preceding theorem will be proved. This time, a path transformation leaving the law of $\e$ invariant suggested by B. Haas, who kindly allowed the author to present it here,  will be used to transform the problem into one concerning deterministic levels in lieu of the random ones, while a limit theorem for the normalized Brownian excursion similar to the one given  by Jeulin in \cite{jeulinsmg} will allow us to conclude. The proof of invariance of the law of $\e$ by Haas' path transformation, which is due to the author, will be done by means of William's description of the It\^o measure and his reversibility theorems for the three-dimensional Bessel processes. Biane provides in \cite{bianevervaat} an explanation by continuous-time methods of invariance  of the law of $\e$ under Vervaat's transformation (\cite{vervaattrans}) similar to the one we shall give. It might be of interest to relate our limit process to the fragmentation with immigration processes introduced by Hass in \cite{haasEquilibrium}: if $FI_t$ denotes the decreasing sequence of the lengths of the bounded connected components of $\set{s\geq 0:R_s>t}$, then $FI$ evolves as a self-similar fragmentation with immigration, in which a particle of mass $x$ arrives at the system at rate $1/\sqrt{2\pi}x^{3/2}$ (no two particles arriving at the same time) and breaks appart like a Brownian height fragmentation independently of the other particles in the system. Whether this relationship is useful for translating the known results of the fragmentation with immigration to our limit process is unknown to the author at present. A more immediate consequence of Theorem \ref{asymptotics} is:
\begin{cor}
\label{corasymptotic}
Let $M_t$ be the Lebesgue measure of $\hat{F}^2_t$ and $H_t$ be the Lebesgue measure of the interval of $\hat{F}^2_t$ that contains $S$. Then  $r\mapsto H_{rt}/t^2$ converges weakly as $t\to 0+$ to the increasing self-similar additive process with Laplace transform $q\mapsto \paren{\sqrt{2q}/\imf{\sinh}{\sqrt{2q}}}^2$ at time $1$ and $M_{t}/t^2$ has a limiting law with Laplace transform $q\mapsto \paren{1/\imf{\cosh}{\sqrt{2q}}}^2$.
\end{cor}

The laws encountered in the corollary belong to the two infinitely divisible families studied in \cite{bpyzeta} and \cite{hyperbolic}. An anonymous referee remarks that Corollary \ref{corasymptotic} suggests an iterated logarithm law. Following this suggestion, we find:
\begin{teo}
\label{lil}
Almost surely,\begin{esn}
\liminf_{t\to 0+}\frac{\log\log t}{2t^2}M_t=1=\liminf_{t\to 0+}\frac{\log\log t}{2t^2}H_t.
\end{esn}
\end{teo}

The organization of the paper is as follows: in Section \ref{representationtaggedfragment} we prove Proposition \ref{conditionedsubordinator}, in Section \ref{fallingapart} we prove Theorem \ref{teosaltos} and Proposition \ref{localTimeBesselBridge}, finally proving Theorems \ref{asymptotics}, \ref{lil} and Corollary \ref{corasymptotic} in Section \ref{extinction}.

\section{The representation of the tagged fragment}
\label{representationtaggedfragment}

In this section, we shall prove Proposition \ref{conditionedsubordinator}. This will be done, in the Brownian case in subsection \ref{brownianrepresentation}, by means of Bismut's decomposition of It\^o's measure, while it will rely on considerations involving positive self-similar Markov processes, like the mass of a tagged fragment of a self-similar fragmentation, in the general case of the fragmentations $F^\alpha$ in subsection \ref{stablerepresentation}.

To complete the interpretation of $-\sigma^h$ given in Section \ref{intro}, let us see that the opposite of a stable subordinator of index $\beta$ starting at one and conditioned to die at zero actually dies at zero. With the notation already introduced, this is expressed as: $\hat{\p}^h_x$-almost surely, $-\sigma^h_{\zeta^h-}=0$. To this end, let us determine the conditional law of $-\sigma^h_t,t< \zeta$ given $\zeta=a$: since\begin{align*}
&\imf{\hat{\p}^h_x}{\zeta^h\in da}=\frac{\imf{f_a}{x}}{\imf{h}{x}}da
\intertext{the Markov property implies that for decreasing $x_i$ and increasing $t_i$:}
&\imf{\hat{\p}^h_x}{-\sigma^h_{t_1}\in dx_1,\ldots,-\sigma^h_{t_n}\in dx_n, \zeta^h\in da}/\,dx_1\cdots dx_n\,da
\\&=\frac{\imf{f_{t_1}}{x_1-x}\imf{f_{t_2}}{x_2-x_1}\cdots \imf{f_{t_n}}{x_n-x_{n-1}}\imf{f_{a-t_n}}{x_n}}{\imf{h}{x}}
\end{align*}so that a version of the conditional law of $-\sigma^h_t,t< \zeta$ given $\zeta=a$ under $\hat{\p}_x^h$ is that of a bridge of $-\sigma$ between $x$ and $0$ of length $a$. Thanks to Proposition 1 in \cite{markovbridges}, we know that the left-hand limit at $a$ of such a bridge is equal to $0$ almost surely, and this implies that $-\sigma^h_{\zeta-}=0$ $\p_x$-almost surely.

Using the self-similarity of $\sigma$, it follows that $-\sigma^h$ is a positive self-similar Markov process; this fact will be crucial to establishing Proposition \ref{conditionedsubordinator} for all $\alpha\in (1,2]$. However, in the Brownian case, we only need to calculate the finite-dimensional distributions of the tagged fragment and compare them to those of $-\sigma^h$, as we will do in subsection \ref{brownianrepresentation} armed with Bismut's representation of the It\^o measure (found in \cite{revuzyor}) followed by a conditioning by the length. 
\subsection{An analysis under It\^o's measure}
\label{brownianrepresentation}
We shall work under It\^o's measure of positive excursions of Brownian motion denoted $n_+$. It can be described in terms of the law of the normalized Brownian excursion, which is the content of It\^o's description of the It\^o measure, as follows. Let $\paren{E,\mc{E}}$ denote excursion space consisting of continuous functions $\fun{e}{[0,\infty)}{[0,\infty)}$ for which there exists $L=\imf{L}{e}\geq 0$, called the length of the excursion, such that $\imf{e}{t}\neq 0$ iff $0<t<L$, together with the $\sigma$-field generated by the canonical process $X$. This $\sigma$-field is also the one generated by the topology of $E$ when we use a metric for uniform convergence on compact sets. It\^o's description of $n_+$ is that the law of $L$ under $n_+$ admits a density given by $v\mapsto 1/2\sqrt{2\pi v^3}\indi{v\geq 0}$ and that the conditional law of $\paren{e_t}_{t\leq L}$ given $L=v$ is that of a Brownian excursion of length $v$, denoted $\pi^v$, so that it has the law of $\paren{\sqrt{v}X_{s/v}}_{t\in [0,v]}$ under $\pi$ (simplified notation for $\pi^1$). This means that for every bounded and measurable functional $\Phi$ on $E$ and  measurable $\fun{g}{[0,\infty)}{[0,\infty)}$, the following equality holds:
\begin{esn}
\label{itosdescription}
\imf{n_+}{\imf{g}{L}\Phi}=\int_0^\infty\frac{dv}{2\sqrt{2\pi v^3}} \imf{g}{v}\imf{\pi^v}{\Phi}.
\end{esn}

However, since the notion of  tagged fragment involves an independent uniform random variable, we are forced to introduce the measure $\tilde{n}_+$ over $\tilde{E}=[0,\infty)\times E$ given by\begin{esn} 
\imf{\tilde{n}_+}{dt,de}=\frac{1}{L}\indi{t\in (0,L)}\, dt\,\imf{n_+}{de}.
\end{esn}If we define the functions $X$, $U$ and $\zeta$ on $\tilde{E}$ by $\imf{X}{t,e}=e$, $\imf{U}{t,e}=t$ and $\imf{\zeta}{t,e}=e_t$, $X$ will take the place of the excursion $\e$, $U$ will take the place of our independent uniform random variable and $\zeta$ will be the death time of the tagged fragment once we extend the definitions of $F$ and $\chi$ over to $\tilde{E}$ by  taking into account the length $L$ in their definitions and use the process $X$ instead of $\e$ as follows: $F_t=\set{s\in (0,L):X_s>t}$ and $\chi_t$ is the length of the connected component of $F_t$ that contains $U$. As a final preliminary before commencing the proof of Proposition\ref{conditionedsubordinator} in the Brownian case, let us recall Bismut's description of the It\^o measure (cf. \cite[XII.4.7, p.502]{revuzyor}): under $L\cdot\tilde{n}_+$, the law of $\zeta$ is Lebesgue measure on $[0,\infty)$  and conditionally on $\zeta=a$,\begin{esn}
\paren{X_{s\wedge U}}_{s\geq 0}\quad\text{ and }\quad \paren{X_{\paren{L-s}^+\wedge\paren{L-U}}}_{s\geq 0}
\end{esn}are two independent Bessel processes of dimension $3$ processes stopped at their last visit to $a$. Therefore, under $L\cdot\tilde{n}_+$ and conditionally on $\zeta=a$, the tagged fragment  behaves like the process obtained by subtracting the last visit process of the concatenation of two independent Bessel  processes of dimension three on $[0,a]$ its final value; by one of William's time reversal theorems (cf. \cite[VII.4.6, p.317]{revuzyor}), it behaves like the process obtained by subtracting $1/2$-stable subordinator with Laplace exponent $q\mapsto 2\sqrt{2q}$ on the time interval $[0,a]$ its final value. If $T$ is such a subordinator, conditionally on $\zeta=a$, $\chi$ would be equal in law to $\paren{T_a-T_t}_{t\in [0,a]}$ and $L$ would be just $T_a$.
\begin{proof}[Brownian proof of Proposition \ref{conditionedsubordinator}]
Let $T$ be as above and  $f_t$ denote the density of $T_t$, given by\begin{esn}
\imf{f_t}{x}=\frac{\sqrt{2}t}{\sqrt{\pi x^3}}e^{-2t^2/x}.
\end{esn}The considerations of the preceding paragraph allow us to write
\begin{align*}
&\imf{\tilde{n}_+}{\zeta\in da,\chi_{t_1}\in dx_1,\ldots,\chi_{t_n}\in dx_n,L\in dv}/\, da\, dx_1\cdots dx_n\,dv
\\&=\frac{1}{v}\, \imf{f_{t_1}}{v-x_1}\imf{f_{\paren{t_2-t_1}}}{x_1-x_2}\cdots
\imf{f_{\paren{t_n-t_{n-1}}}}{x_{n-1}-x_n}\imf{f_{\paren{a-t_n}}}{x_n}
\end{align*}for decreasing $x_1,\ldots,x_n$ in $[0,v]$ (the tagged fragment decreases in size) and increasing $t_1,\ldots,t_n$ in $[0,a]$; note that the factor $1/v$ comes from the fact that we are not working with $L\cdot \tilde{n}_+$ (as in Bismut's description of $n_+$) but with $\tilde{n}_+$. 

Integrating $a$ out of the right hand side of the last display over the interval $(t_n,\infty)$ gives\begin{align*}
&\imf{\tilde{n}_+}{\chi_{t_1}\in dx_1,\ldots,\chi_{t_n}\in dx_n,L\in dv}/\, dx_1\cdots dx_n\,dv
\\&=\frac{1}{v}\, \imf{f_{t_1}}{v-x_1}\imf{f_{\paren{t_2-t_1}}}{x_1-x_2}\cdots
\imf{f_{\paren{t_n-t_{n-1}}}}{x_{n-1}-x_n}\frac{1}{2\sqrt{2\pi x_n}}.
\end{align*}Conditioning by length, using $\imf{\tilde{n}_+}{\zeta\in dv}=1/2\sqrt{2\pi v^3}$, allows the following conclusion:\begin{align*}
&\imf{\pi^v}{\chi_{t_1}\in dx_1,\ldots,\chi_{t_n}\in dx_n}/\, dx_1\cdots dx_n
\\&=\imf{f_{t_1}}{v-x_1}\imf{f_{\paren{t_2-t_1}}}{x_1-x_2}\cdots
\imf{f_{\paren{t_n-t_{n-1}}}}{x_{n-1}-x_n}\frac{\sqrt{v}}{\sqrt{x_n}}.
\end{align*}The right-hand side of the preceding display portrays the density of the finite-dimensional distributions of the opposite of a $1/2$-stable subordinator with Laplace exponent $q\mapsto 2\sqrt{2q}$ conditioned to die at zero started at $v$.
\end{proof}
\subsection{An analysis through positive self-similar Markov processes}
\label{stablerepresentation}
The definition of the fragmentation $F^\alpha$ is not as simple when $\alpha\in (1,2)$ as in the $\alpha=2$  case previously introduced (recall that $-1/\alpha$ stands for the index of the self-similar fragmentation) because its construction depends on the so-called height process. However, for our needs, concentrating on a tagged fragment of $F^\alpha$ will be enough. 

The tagged fragment associated to $F^\alpha$, denoted by $\chi^\alpha$, is a self-similar Markov process that is absorbed continuously at zero in finite time $\zeta^\alpha$. Thanks to the Lamperti transformation\footnote{For information and further references regarding self-similar Markov process, the Lamperti transformation and its relationship to exponential functionals of L\'evy processes, see the recent survey \cite{surveyexponentialfunctionals}.} it is associated to a subordinator $\xi^\alpha$ whose L\'evy measure has been explicitly calculated, in \cite{ssfrag} and \cite{heights}, and is given by\begin{equation}
\label{medlevybertoin}
x\mapsto\sqrt{\frac{2}{\pi}}\frac{e^x}{\paren{e^x-1}^{3/2}}
\end{equation}for the Brownian case (a multiple of (11) in \cite{ssfrag}, as explained there) and, recalling that $\beta=1-1/\alpha$,\begin{equation}
\label{medlevymiermont}
x\mapsto\frac{\beta}{\imf{\Gamma}{2-\beta}}\frac{e^x}{\paren{e^x-1}^{1+\beta}}
\end{equation}for $\alpha\in (1,2)$ (the display after (12) in \cite{heights}). The difference in the constant appearing is due to the fact that one uses the normalized Brownian excursion and not a scaled one corresponding to the height process of a Brownian excursion in Bertoin's construction of the fragmentation. We will now proceed with the proof of Proposition \ref{conditionedsubordinator}.
\begin{proof}[Proof of Proposition \ref{conditionedsubordinator}]
The Lamperti transformation takes a L\'evy process $\xi$ and a real number $a$ into the self-similar Markov process that starts at one given implicitly by\begin{esn}
T_{\int_0^t\imf{\exp}{a \xi_s}\, ds}=e^{\xi_t}.
\end{esn}The index of self-similarity of $T$, as defined in \cite{lamp72}, is then $1/a$. When applied to a subordinator $\xi$ and successively with $a$ and $-a$ for a positive $a$, it gives rise to two different processes, denoted by $T$ and $\hat{T}$ respectively, which  are nevertheless related to each other by duality of their resolvent operators with respect to Lebesgue measure on $(0,\infty)$, as shown in \cite{entrancelaws}. When $T$ is a $\beta$-stable subordinator with Laplace exponent $q\mapsto Cq^\beta$, the associated subordinator $\xi$ has L\'evy measure\begin{esn}
x\mapsto \frac{\beta C}{\imf{\Gamma}{1-\beta}}\frac{e^x}{\paren{e^{x}-1}^{1+\beta}}
\end{esn}which coincides with \eqref{medlevybertoin} when $\beta=1/2$ and $C=2\sqrt{2}$ and with \eqref{medlevymiermont} when $\beta\in(1/2,1)$ and $C=\imf{\Gamma}{1-\beta}/\imf{\Gamma}{2-\beta}$; the Lamperti transformation should be applied to $\xi$ with $a=\beta$ to obtain $T$. It follows that the tagged fragment of $F^\alpha$ (we had denoted it by $\chi^\alpha$) is in resolvent duality with a $\beta$-stable subordinator. Since $T_t$ admits a  density  $f_t$ then, as argued in \cite{entrancelaws}, when we view $\chi^\alpha$ time-reversed from its death time, it behaves as $T$ started at zero and conditioned to die at one via Doob's transformation with the excessive function $\imf{h}{x}=\int_0^\infty\imf{f_t}{1-x}\, dt$. Nagasawa's theorem on time-reversal then allows us to conclude that $F^\alpha$ has the same law as $-T$ started at one and conditioned to die at zero via $\imf{\hat{h}}{x}=\int_0^\infty\imf{f_t}{x}\, dt$.
\end{proof}
The  computations of the L\'evy measure of the subordinator $\xi$  associated to the death time of the tagged fragment of $F^\alpha$ performed in \cite{ssfrag} and \cite{heights} were based on the fact that one can express the density of the death time of the tagged fragment in terms the density  of a $\beta$-stable subordinator. With the notation introduced in the preceding proof, the density of the death time of the tagged fragment is $t\mapsto \imf{f_t}{x}/\imf{h}{x}$, where the constant $C$ chosen in terms of $\beta$ as mentioned during the course of the proof. This last expression should suffice to convince oneself of the validity of Proposition\ref{conditionedsubordinator} because of the following result:
\begin{lem}
\label{determinedbydeathtime}
The distribution of a decreasing and positive self-similar Markov process that is absorbed at zero in finite time is determined by its index and the law of its absorption time.
\end{lem}
\begin{proof}
By self-similarity, it suffices to consider the case when the given process starts at one. If $\zeta$ denotes the absorption time of a decreasing and positive self-similar Markov process starting at one which is absorbed at zero in finite time  obtained by applying the Lamperti transformation to a subordinator $\xi$, then there exists $\delta<0$ (one over the self-similarity index) such that $\zeta$ has the same law as the exponential functional\begin{esn}
A_\infty=\int_0^\infty\imf{\exp}{\delta\xi_s}\, ds.
\end{esn}On the other hand, if $\phi$ is the Laplace exponent of $\xi$, then formula (4) in \cite{surveyexponentialfunctionals} used with $q=0$ gives us\begin{esn}
\esp{\zeta^k}=\esp{A_{\infty}^k}=\frac{k!}{\prod_{i=1}^k \imf{\phi}{-\delta i}}.
\end{esn}It follows that the sequence of moments $\paren{\esp{\zeta^k}}_{k\in\na}$ determines the sequence $\paren{\imf{\phi}{-\delta i}}_{i\in\na}$. However, the second sequence determines the moments of the bounded random variable $\imf{\exp}{\delta \xi_t}$, so that it determines its law, hence that of $\xi_t$. Finally, it suffices to note that the distribution of $\xi_t$ and the self-similarity index determine the distribution of the self-similar Markov process we started with. 
\end{proof}
\section{The falling apart of the tagged fragment and fragmentation by ancestral line obliteration}
\label{fallingapart}
In this section, we shall prove Theorem \ref{teosaltos} and Proposition \ref{localTimeBesselBridge}. First,  the Brownian case of Theorem \ref{teosaltos} will be considered in subsection \ref{geometry}  using  a path transformation relating the normalized Brownian excursion and the Brownian bridge, and known results on the distribution of the ranked length of excursions of the Brownian bridge away from zero. Then, we shall see how these results tie up in the construction of another self-similar fragmentation from the normalized Brownian excursion. Finally, the proof for the general case will be shown to be the consequence of  our representation of the tagged fragment of $F^\alpha$ contained in Proposition \ref{conditionedsubordinator}, which allows us to calculate its conditional distribution given death-time and relate it to a stable subordinator, and known results on size-biased sampling of the jumps of subordinators.
\subsection{A visual argument for the Brownian case}
\label{geometry}
The Brownian interpretation of Theorem \ref{teosaltos} (that is, using the fragmentation $F^2$) is  quite visual and depends on a path transformation, introduced by Bertoin and Pitman,  between the normalized Brownian excursion and the reflected Brownian bridge which can be stated as follows (cf. \cite[Theorem 3.2]{pathtrans}): define $K^U=\paren{K^U_s}_{s\in [0,1]}$ by\begin{esn}
K^U_s=\begin{cases}
\min_{s\leq u\leq U}\e_u&\text{for $s\in [0,U]$}\\
\min_{U\leq u\leq s}\e_u&\text{for $s\in [U,1]$}.
\end{cases}
\end{esn}Then the process $b=\e-K^U$  is the absolute value of a Brownian bridge between $0$ and $0$ of length $1$. Let us note, however that the lengths of the excursions of $\e$ above $K^U$ are in one to one correspondence with the jumps of $\chi$. Since the excursions of $\e$ above $K^U$ are precisely the excursions of $b$ away from zero, we conclude that the decreasing sequence of the jumps of $\chi$ has the same law as the decreasing sequence of the lengths of excursions of a Brownian bridge away from zero. By Proposition 7 in \cite{poissondirichlet}, this is the Poisson-Dirichlet law with parameters $(1/2,1/2)$. This proves Theorem \ref{teosaltos} for $\alpha=2$.

The same type of analysis can be put to use in the construction of another fragmentation process. Define $b^0=\e$  and suppose that $\paren{U_i}_{i\geq 1}$ are independent (between themselves and $\e$) and uniformly distributed random variables.  For $n\geq 1$ construct $b^n$  as follows, and set $V_n=\set{s\in (0,1):b^n_s>0}$:  let  $\paren{a_{n-1},b_{n-1}}$ be the connected component of $V_{n-1}$ that contains $U_n$ , $\paren{K^n_s}_{s\in [0,1]}$ be given by\begin{esn}
K^{n}_s=
\begin{cases}
0&\text{if $s\not\in (a_{n-1},b_{n-1})$}\\
\min_{a_{n-1}\leq s\leq U_n} b^{n-1}_s&\text{if $a_{n-1}\leq s\leq U_n$}\\
\min_{U_n\leq s\leq b_{n-1}} b^{n-1}_s&\text{if $U_n\leq s\leq b_{n-1}$}\\
\end{cases},
\end{esn}and $b^n=b^{n-1}-K^n$.  To construct a self-similar interval fragmentation, let $N$ be a Poisson process independent of $\e$ and $\paren{U_i}_{i\geq 1}$, and  set $F^o_t=V_{N_t}$. This fragmentation, which has self-similarity index 1, erosion coefficient zero and dislocation measure equal to  $\imf{\pd}{1/2,1/2}$, shall be termed by ancestral line obliteration and we shall dwell next  on its interpretation and on a computation that can be performed with it.

We recall that the compact real tree $\paren{\tau_f,d_f,\rho_f}$ represents a genealogy coded by the function $f$ as mentioned in Section \ref{intro}. The random trees we shall be interested in are $\paren{\tau_{b^n},d_{b^n},\rho_{b^n}}$.  To continue the analogy presented in Section \ref{intro}, consider a the tree coded by a continuous function $f$ and let us note that the common ancestor of every element of $\tau_f$ is $\rho_f$, the most recent common ancestor of $s_1$ and $s_2$ is the equivalence class of any $r\in[s_1\wedge s_2,s_1\vee s_2]$ such that $\imf{m_f}{s_1,s_2}=r$ and the line of descent traced  from the ancestor $\rho_f$ up to the equivalence class of $s$ consists of equivalence classes of elements $r\in [0,1]$ such that $\imf{f}{r}=\imf{m_f}{r,s}$. To concatenate with our fragmentation $F^o$, let us note that $K^n_s=\imf{m_{b^{n-1}}}{U_n,s}$ and so $b^{n-1}-K^{n}$ represents the coding function for a tree that redefines the genealogy of $\tau_{b^{n-1}}$ by not taking into account the equivalence class of $U^n$ (in $b^{n-1}$) and all its ancestors up to the root. The interpretation of this transformation between continuous functions and their associated trees does not appear to be reported elsewhere. We refer to Figure \ref{obliteration} for a visual account of this procedure. 
\begin{figure}
\includegraphics{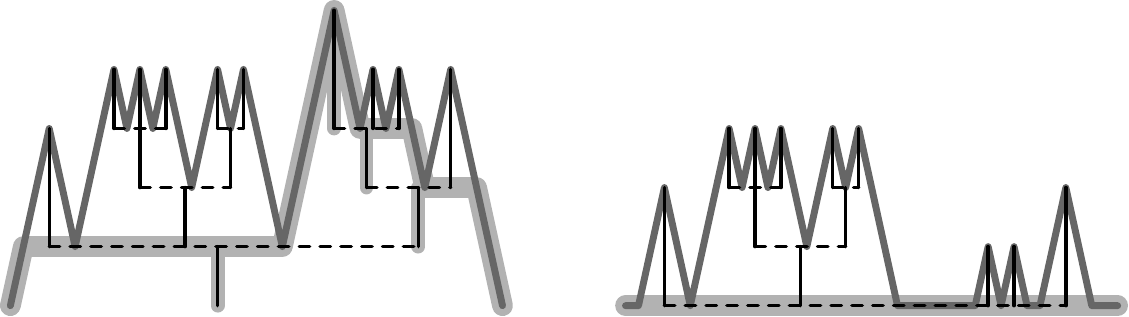}
\caption{Trees coded by continuous functions and obliteration of ancestral lines.}
\label{obliteration}
\end{figure}

To end this subsection, let us describe the law of the decreasing sequence of masses of the components  of $V_n$, which we shall denote $m_n$. To do this, let us note that $m_n$ is  obtained by taking a size-biased pick from $m_{n-1}$ (the size of the component of $V_{n-1}$ that contains $U_n$) and fragmentating it using a $\imf{\pd}{1/2,1/2}$-distribution. So, $m_{n}$ is obtained as the result of applying the fragmentation operator $\mathbf{Frag_{1/2}}$ of \cite{coagfrag} to $m_{n-1}$. Since $m_1$ has a $\imf{\pd}{1/2,1/2}$ distribution, Theorem 3.1 in the last reference implies that for $n\geq 1$,  $m_n$ has a $\imf{\pd}{1/2,n-1/2}$ distribution.
\subsection{A computational argument for the general case}
\label{computation}
The aim of this subsection is to establish Theorem \ref{teosaltos}. Our strategy  will be  to analyze the implications of Proposition \ref{conditionedsubordinator} by computing the conditional law of the tagged fragment given its death time, and its relationship to stable subordinators. Then we shall use this conditional law in conjunction with formulae describing  size-biased sampling of the jumps of subordinators to conclude. 

We shall use the framework and notation considered in the introduction to section \ref{representationtaggedfragment}. One conclusion of the introduction is that under $\hat{\p}^{h}_1$ and conditionally on $\zeta^h=a$, $-\sigma^h$ is a bridge of $-\sigma$ from $1$ to $0$ of length $a$, so the same result follows for the tagged fragment.  Also, the bridge of $-\sigma$ from $x$ to $y$ of length $v$ coincides with the opposite of that of $\sigma$ between $-x$ to $-y$ , so that the sizes of the jumps, in absolute value, are the same for both processes. We shall now prove Theorem \ref{teosaltos}.
\begin{proof}[Proof of Theorem \ref{teosaltos}]
Let us recall that the two-parameter Poisson-Dirichlet distribution (see the survey \cite{poissondirichlet} for further information and references), denoted by $\imf{\pd}{\beta,\theta}$ (we will think of $\beta$ as a fixed parameter) for  $\theta>-\beta$, is  a probability law on the space of decreasing sequences $v=v_1>v_2>\cdots>0$ such that $\sum_i v_i=1$, which is characterized by a property of their size-biased permutations:  $V=\paren{V_1,V_2,\ldots}$ has a $\imf{\pd}{\beta,\theta}$ distribution iff for a size biased permutation $\tilde{V}$ of $V$, the random variables defined implicitly by $\tilde{V}_1=Y_1$, $\tilde{V}_n=\paren{1-Y_1}\cdots\paren{1-Y_{n-1}}Y_n$ are independent and $Y_n$ has a Beta distribution with parameters $\paren{1-\beta,\theta+n\beta}$. A size biased permutation of $V$ is another sequence $\tilde{V}$ such that\begin{align*}
\probac{\tilde{V}_1=V_i}{V}&=V_i
\intertext{and}
\probac{\tilde{V}_{n+1}=V_i}{V,\tilde{V}_1,\ldots,\tilde{V}_n}&=\frac{V_i}{1-\tilde{V}_1-\cdots-\tilde{V}_n}\indi{V_i\neq \tilde{V}_j,1\leq j\leq n}.
\end{align*}One of the objectives of \cite{permanpitmanyor} is to construct the tools necessary for the analysis of size-biased permutations of the jumps of subordinators.  Let us start the process $-\sigma^h$ at level one, simplifying notation by stipulating that $\p_1=\p$, and consider  the decreasing sequence $V$ of the absolute values of the jump sizes of $-\sigma^h$ on $[0,\zeta]$, which will then sum up to one, and a size-biased permutation  $\tilde{V}$ of $V$ giving rise to the sequence $Y$ as before. We shall use the  the density of the L\'evy measure of $\sigma$, given by $\imf{\rho}{x}=\beta C/\imf{\Gamma}{1-\beta}x^{\beta+1}$ to define the function $\imf{\Theta}{x}=x\imf{\rho}{x}$. The discussion of the preceding paragraph and formula (2.d) of Theorem 2.1 in \cite{permanpitmanyor}, using the notation $\bar{x}=1-x$, allow the following:\begin{align*}
\proba{Y_1\in dx_1,\ldots,Y_n\in dx_n, \zeta\in da}
=v^n\imf{\Theta}{x_1}\imf{\Theta}{\bar{x}_1x_2}\cdots\imf{\Theta}{\bar{x}_1\cdots\bar{x}_{n-1}x_n}\imf{f_a}{\bar{x}_1\cdots\bar{x}_n}
\frac{1}{\imf{h}{1}}.
\end{align*}Now, we shall use the scaling identities of $f_t$ to integrate $a$ out of the last expression. Namely, since\begin{esn}
\imf{f_t}{y}=\frac{1}{y}\imf{f_{ty^{-\beta}}}{1},
\end{esn}we get\begin{align*}
&\proba{Y_1\in dx_1,\ldots,Y_n\in dx_n}
\\&=\frac{\beta^n}{\imf{\Gamma}{1-\beta}^n}\esp{\zeta^n}\imf{\Theta}{x_1}\imf{\Theta}{\bar{x}_1x_2}\cdots\imf{\Theta}{\bar{x}_1\cdots\bar{x}_{n-1}x_n}
\\&=\frac{\beta^n}{\imf{\Gamma}{1-\beta}^n}\esp{\zeta^n}C^n \paren{x_1\cdots x^n}^{\beta-1}\bar{x}_1^{2\beta-1}\cdots\bar{x}_n^{\paren{n+1}\beta-1}.
\end{align*}Now, let us note that the last expression in the preceding display does not depend on $C$, which can be seen either by direct analysis of the law of $\zeta$ considering the scaling identity of $f_t$, or by the fact that the left-hand side of the first equality in the preceding display represents a probability density. The conclusion is that $Y_1,\ldots,Y_n$ are independent and $Y_n$ has a Beta distribution with parameters $1-\beta$ and $\beta+n\beta$, so that the sequence of jumps of $\chi$ in decreasing order has the $\imf{\pd}{\beta,\beta}$ distribution.
\end{proof}

The Poisson-Dirichlet distribution of parameters $\paren{\beta,\beta}$ arises as the distribution of the ranked lengths of excursions of Bessel bridges of dimension $\delta=2\paren{1-\beta}$ starting and ending at zero (cf. \cite[Proposition 7]{poissondirichlet}). Since the inverse local time at zero of a Bessel process of dimension $\delta$ starting at zero is a stable subordinator of index $\beta$, it is natural to search for a similar representation for the inverse local time of our Bessel bridge; it turns out that inverse local time is a stable subordinator of index $\beta$ starting at zero and conditioned to die at $1$ (through a Doob transformation via the potential density as described in Section \ref{representationtaggedfragment}); this is the content of Proposition  \ref{localTimeBesselBridge} whose proof is as follows:
\begin{proof}[Proof of Proposition \ref{localTimeBesselBridge}]
Let $\p^\delta$  denote the law of a Bessel process of dimension $\delta$ starting at zero and $\p^\delta_{1}$ be the law of a Bessel bridge of dimension $\delta$ and length one starting and ending at zero. (For a general account of the theory of bridges of Markov processes, see \cite{markovbridges}.) Using the explicit representations of the transition densities of Bessel processes (in terms of modified Bessel functions of the first kind) one can prove that, for $s<1$, we have the following relationship between $\p^\delta_1$ and $\p^{\delta}$ (where $X$ stands for the canonical process and $\F_s=\sag{X_u:u \leq s}$):\begin{esn}
\p_1^{\delta}|_{\F_s}=\paren{\frac{1}{1-s}}^{\beta}e^{-\frac{X_s^2}{2\paren{1-s}}}\cdot \p^\delta|{\F_s}.
\end{esn}Let $\tau$ denote the inverse local time at zero, where the local time is taken in the sense of regenerative sets (semimartingale local time vanishes as explained in \cite[XI.1.5, p.442]{revuzyor}). Just as in \cite[VIII.1.3, p.326]{revuzyor}, we can extend the preceding equality to the stopping times $\tau_s$ on the set $\set{\tau_s<\infty}$ so that\begin{esn}
\p_1^{\delta}|_{\F_{\tau_s}}=\paren{\frac{1}{1-\tau_s}}^{\beta}\cdot \p^\delta|{\F_{\tau_s}}.
\end{esn}Since $\tau$ is a stable subordinator of index $\beta$ under $\p^\delta$ ($\tau_s$ has density $f_s$), it follows that under $\p^\delta_1$, $\tau$ is Markovian and its transition density from $x$ to $y$ in $s$ units of time is $\imf{f_s}{y-x}\paren{\paren{1-x}/\paren{1-y}}^\beta$, so that it  is a $\beta$-stable subordinator conditioned to die at $1$.
\end{proof}
\section{Asymptotics at extinction}
\label{extinction}
In this section we shall prove Theorem \ref{asymptotics}. The reader is asked to recall the framework introduced in the introduction in order to state it. A point that was not discussed there was the proper topology on the set of open subsets of $\re$ to be able to talk about weak convergence. To introduce it, consider first the following metric on $\mc{V}$ introduced in \cite{ssfrag}: for any $V\in\mc{V}$, let $\chi_V$ be the continuous function on $[0,1]$ given by $\imf{\chi_V}{x}=\imf{d}{x,[0,1]\setminus V}$, and for $V_1,V_2\in\mc{V}$, set $\imf{d_\mc{V}}{V_1,V_2}=\|\chi_{V_1}-\chi_{V_2}\|_{\infty}$. The distance between $V_1$ and $V_2$ is equal to the Hausdorff distance between $V_1^c$ and $V_2^c$ (where complementation is with respect to $[0,1]$) and it turns $\mc{V}$  into a separable compact metric space; we shall therefore speak of Hausdorff's topology on $\mc{V}$. For an open subset $V\subset\re$, let $\mc{V}^V$ be the set of open subsets of $V$.  Bertoin's metric on the $\mc{V}^{(0,1)}$ discussed in the introduction to this section  can be immediately extended to a metric $d_V$  for $\mc{V}^V$, when $V$ is a  bounded open set; it turns this space  into a compact and separable metric space, hence a Polish one. If $V$ is  an unbounded (open) set, we can define\begin{esn}
d_V=\sum_{n\in \z}\frac{d_{V\cap (n,n+1)}}{2^n}
\end{esn}so that $\mc{V}^V$ is again a Polish space. It is in this sense that we will consider random open subsets of $\re$; choosing a bounded metric giving the same topology of $\re$ in the definition of the Hausdorff distance would have the same effect. To discuss measurability issues, we shall use the following multiplicative system of functions of $\mc{V}$ generating its Borel \sa. The family of functions is $\mc{M}=\set{e^{-f}:f\in\mc{D}}$ where
\begin{align*}
\mc{D}=\set{\fun{f}{\mc{V}}{\re_+}:\text{there exists a positive measure $\mu\ll\lambda$ on $(0,1)$ such that $\imf{f}{V}=\imf{\mu}{V}$}},
\end{align*}and $\lambda$ is Lebesgue measure.
\begin{lem}
\label{multiplicative}
The classes $\mc{D}$ and $\mc{M}$  both generate $\mc{B}_{\mc{V}}$.
\end{lem}
We can use the preceding lemma to see that the $\mc{V}^{(0,1)}$-valued  variable $\hat{F}^2_t$ is measurable, and since, for any open subset  $V$  of $\re$, the inclusion from $\mc{V}^V$ into $\mc{V}^\re$ is continuous, it is also a random variable with values in $\mc{V}^\re$. In fact, we shall see first that if $T$ is a random variable with values in $[0,\infty)$, then\begin{esn}
F^2_T=\set{s\in (0,1):\e_s>T}
\end{esn}is a $\mc{V}$-valued random variable. (This implies $\hat{F}^2_t$ is also a random variable.) To do that, we note that thanks to display (2) in \cite{ssfrag}, $t\mapsto F^2_t$ is right-continuous on $[0,\infty)$, so that it suffices to prove that $F^2_t$ is a random variable for every deterministic $t\in [0,\infty)$. By Lemma \ref{multiplicative} it suffices to prove that for every measure $\mu$ on the Borel sets of  $(0,1)$, $\imf{\mu}{F^2_t}$ is measurable. Since\begin{esn}
\imf{\mu}{F^2_t}=\int_0^1\indi{\e_s>t}\,\imf{\mu}{ds}
\end{esn}and the trajectories of $\e$ are continuous, we obtain the measurability of $\imf{\mu}{F^2_t}$ and as a consequence, that of $F^2_t$.  A similar argument implies that $\set{t\in\re:Z_t<1}$ is a $\mc{V}^\re$-valued random variable. Let us note that, if $R_{U}$ denotes the restriction map $V\mapsto V\cap U$, then a sequence $\paren{\mu_n}_{n\in\na}$  of probability laws on $\mc{V}^\re$ converges in distribution to $\mu$ iff for every $i\in\z^+$, $\mu_n\circ R_{(-i,i)}^{-1}$ converges in distribution to $\mu\circ R_{(-i,i)}^{-1}$. This is a direct consequence of the fact that, with the distance $d_\re$ defined on $\mc{V}^\re$, a subset $A$ of $\mc{V}^\re$ is compact if and only if $\set{V\cap \paren{-i,i}:V\in A}$ is compact for every $i\in\z^+$.

Theorem \ref{asymptotics} will be proved by the  use of a path transformation to substitute the random time $M-t$ by $t$, and then we shall  vary the length of the excursion instead of the parameter $t$, since a result concerning the Brownian excursion of length $v$ when $v$ tends to $\infty$ can be readily applied. 

We shall now provide a path transformation of the normalized Brownian excursion that leaves its distribution invariant and which translates our problem into one involving initial times rather than  the time $M$ of extinction of $F^2$: if we let\begin{esn}
\e^S_t=
M-\e_{\paren{S+t}\mod 1}
\end{esn}then we have
\begin{pro}
\label{rootchange}
$\e^S$ has the same law as $\e$.
\end{pro}This path transformation of the normalized Brownian excursion was suggested  by B. Haas in a private communication and it is illustrated in Figure \ref{haastrans}.
\begin{figure}
\includegraphics{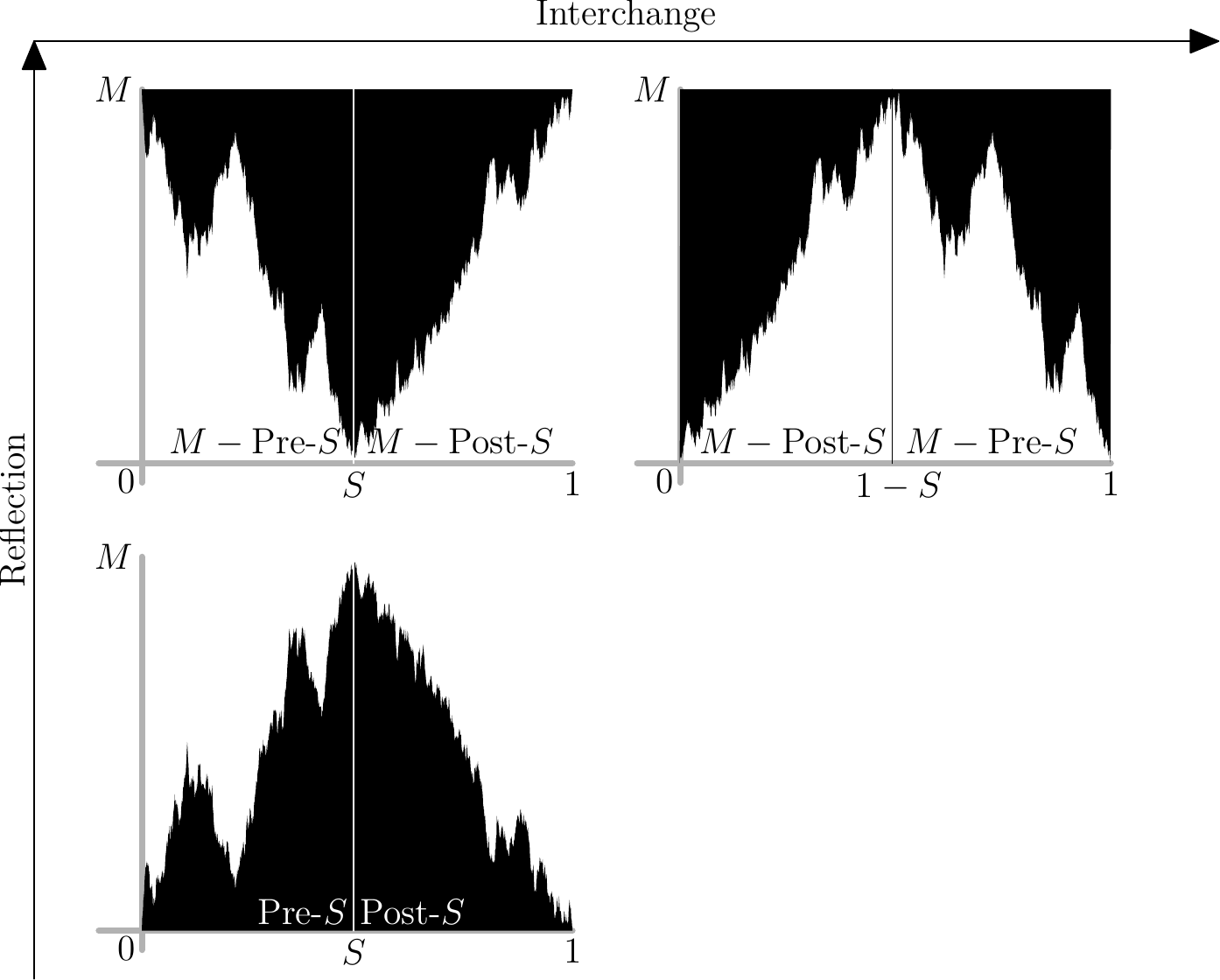}
\caption{Haas' path transformation of the normalized Brownian excursion.}
\label{haastrans}
\end{figure}

Note that $\paren{\hat{F}^2_t-S}/t^2$ can be obtained from  $\e^S$ below level $t$ and the random time $S$. The precise representation depends strongly  on the value of $S$. However, on the event $nt^2<S<1-nt^2$, which is the whole space for fixed $n$ as $t\to 0+$, we have the identity\begin{esn}
\imf{R_{\paren{-n,n}}}{\frac{1}{t^2}\paren{\hat{F}^2_t-S}}= -\set{s\geq 0:\e^S_{st^2}<t,s<n}\cup \set{s\geq 0:\e^S_{1-st^2}<t,s<n}.
\end{esn}Therefore, we shall prove Theorem \ref{asymptotics} by verifying that the right-hand side in the preceding display converges in law to the limit we have stipulated. Since, assuming Proposition \ref{rootchange}, $\paren{\e^S_{st^2}/t}_{s\in (0,1)}$ has the  law of a Brownian excursion of length $1/t^2$ (which was denoted $\pi^{1/t^2}$),  denoting by $X$ be the canonical process on excursion space, we have that\begin{align*}
&-\set{s\geq 0:\e^S_{st^2}<t,s<n}\cup \set{s\geq 0:\e^S_{1-st^2}<t,s<n}
\intertext{has the same law as}
&- \set{s\in (0,n):X_{s}<1}\cup \set{s\in (0,n):X_{1/t^2-s}<1}
\end{align*}under $\pi^{1/t^2}$, at least for $n<1/t^2$. Inspired by a result of Jeulin (cf.  \cite[Th. 6.41 p. 127]{jeulinsmg}) we will prove in subsection \ref{jeulin} that  if $F$ and $G$ are  bounded measurable functionals depending on $\paren{X_s}_{s\in (0,n)}$, and if we let $\tilde{X}_t=X_{v-t}$, then\begin{equation}
\label{ourversionjeulin}
\imf{\pi^v}{\imf{F}{X}\imf{G}{\tilde{X}}}\to_{v\to\infty}\imf{\p_0^3}{F}\imf{\p_0^3}{G}.
\end{equation}We have seen why $\fund{\psi^n_t}{f}{\set{s\in (0,n):\imf{f}{s}>t}}$ is measurable, and since $t\mapsto \psi^n_t$ is right-continuous, then $f\mapsto \paren{\set{s\in (0,s):\imf{f}{s}>t}}_{t\geq 0}$ is a measurable process. The asymptotic identity \ref{ourversionjeulin} and the preceding discussion imply Theorem \ref{asymptotics}, as long as we can be convinced of the validity of Proposition \ref{rootchange}.

Regarding Corollary \ref{corasymptotic}, it suffices to remark that the size of the connected component of the set $\set{s\in\re:Z_s<r}$ which contains zero is equal to the sum of the hitting times of level $r$ by two independent three-dimensional Bessel processes starting at zero, which share $q\mapsto \sqrt{2q}/\imf{\sinh}{\sqrt{2q}}$ as a Laplace transform when $r=1$; the process of hitting-times of a Bessel process is a self-similar increasing additive process. On the other hand, the Lebesgue measure of $\set{s\in\re:Z_s<1}$ is the sum of the occupation times of $(0,1)$ of these independent Bessel processes and, thanks to the Ciesielski-Taylor identity for example (which equals their law to that of the first exit from $(-1,1)$ by a Brownian motion starting at zero) they have a common Laplace transform given by $q\mapsto 1/\imf{\cosh}{\sqrt{2q}}$.

Before commencing the proof of Proposition \ref{rootchange} (in Subsection \ref{haas}) and  the asymptotic relationship \eqref{ourversionjeulin} (in Subsection \ref{jeulin}), let us turn to the proof of Lemma \ref{multiplicative}.
\begin{proof}[Proof of Lemma \ref{multiplicative}]
It suffices to see that $\sag{\mc{D}}=\mc{B}_{\mc{V}}$.

As a consequence of Lemma 2 in \cite{ssfrag}, we see that for every measure $\mu$ on $\mc{B}_{(0,1)}$ absolutely continuous with respect to Lebesgue measure, the function $V\mapsto \imf{\mu}{V}$ is continuous, hence $\mc{B}_{\mc{V}}$\nbd measurable, implying $\sag{\mc{D}}\subset\mc{B}_{\mc{V}}$. To verify the converse inclusion, we note that the definition of $d_\mc{V}$ implies $\mc{B}_{\mc{V}}=\sag{\chi}$ where $\chi$ is the function given by $V\mapsto \chi_V$. We will finish the proof by verifying that $\chi$ is $\sag{\mc{D}}$\nbd measurable. As the Borel subsets of the space of continuous functions on $[0,1]$ equipped with the uniform norm are generated by the projections $f\mapsto \imf{f}{t}$ for $t\in (0,1)$, the asserted measurability for $\chi$ will follow if we verify that for every $t\in (0,1)$, the function $V\mapsto \imf{\chi_V}{t}=\imf{d}{t,[0,1]\setminus V}$ is $\sag{\mc{D}}$\nbd measurable. However, for every  $t\in (0,1)$, $\imf{\chi_V}{t}\leq t\wedge(1-t)$ and for every $\eps\in (0,t\wedge (1-t))$  we can define the measure $\mu$ on $(0,1)$ as Lebesgue measure concentrated on $(t-\eps,t+\eps)$ for which the following holds:\begin{esn}
\set{\imf{d}{t,V^c}\geq \eps}=\set{\imf{\mu}{V}=2\eps}\in\sag{\mc{D}}.
\end{esn}
\end{proof}
\subsection{The transformation $\e\mapsto \e^S$}
\label{haas}
The aim of the following paragraphs is to show how Proposition \ref{rootchange} can be deduced from one of  William's time reversal results  relating Brownian motion killed when it reaches zero and the three-dimensional Bessel process on one hand and It\^o's and William's descriptions of the It\^o measure on the other. 

Let $n_+$ be the It\^o measure of positive excursions of Brownian motion  introduced in subsection \ref{brownianrepresentation}. We shall keep the notation.  
The reader is asked to recall It\^o's description of the It\^o measure since we shall perform a conditioning by the length on $n_+$. To carry out this program, we will also need William's description of the It\^o measure, which is the following. Let $\fun{M}{E}{\re_+}$ denote the height of the excursion, given by $M=\sup_{s\geq 0}X_s$. Then the image law of $M$ under $n_+$ admits the density $m\mapsto \indi{(0,\infty)}/2m^2$ and the conditional law of $X$ under $n_+$ given $M=m$ is that the pasting together of two independent three-dimensional Bessel processes started at zero and stopped when they reach level $m$, one of them concatenated in reverse time after the other. Now, let us recall the following time-reversal result, for which the reader is referred to \cite[VII.4.8]{revuzyor}: if $R$ is a three-dimensional Bessel process starting at zero, $b>0$ and $T_b$ is the hitting time of $b$ by $R$, then $\paren{X_{T_b-t}}_{0\leq t\leq T_b}$ and $\paren{b-X_{t}}_{0\leq t\leq T_b}$ have the same law. 

With these preliminaries, let us commence the proof of Proposition \ref{rootchange}. Let $S$ be the instant in which $X$ attains its maximum and define, as for the normalized Brownian excursion, $X^S$ by\begin{esn}
X^S_t=M-X_{t+S\mod L}.
\end{esn}By using William's description of the It\^o measure and his time-reversal result, we see that under $n_+$ conditionally on $M=m$, $\paren{X^S,L}$ has the same law as $\paren{X,L}$, and so the same holds under $n_+$. If $\fun{g}{\re^n}{\re}$ is bounded and continuous, $\fun{f}{(0,\infty)}{\re}$ is a positive measurable function, and $0\leq t_1\leq \cdots\leq t_n$ we obtain the equality\begin{equation}
\label{testfunction}
\imf{n_+}{\imf{g}{X^S_{t_1},\ldots,X^S_{t_n}}\imf{f}{L}}
=\imf{n_+}{\imf{g}{X_{t_1},\ldots,X_{t_n}}\imf{f}{L}}.
\end{equation}However, by It\^o's description of the It\^o measure, the left hand side of the preceding display equals\begin{esn}
\int_{t_n}^\infty dv\, \frac{1}{2\sqrt{2\pi v^3}}\imf{f}{v}\imf{\pi^v}{\imf{g}{X^S_{t_1},\ldots,X^S_{t_n}}}
\end{esn}while the right-hand side equals\begin{esn}
\int_{t_n}^\infty dv\, \frac{1}{2\sqrt{2\pi v^3}}\imf{f}{v}\imf{\pi^v}{\imf{g}{X_{t_1},\ldots,X_{t_n}}}.
\end{esn}Because the equality in \eqref{testfunction} is valid for any positive measurable function $f$, we  conclude from the weak continuity of $v\mapsto \pi^v$  that\begin{esn}
\imf{\pi^v}{\imf{g}{X^S_{t_1},\ldots,X^S_{t_n}}}=\imf{\pi^v}{\imf{g}{X_{t_1},\ldots,X_{t_n}}},
\end{esn}for all $v>t_n$, so that $\pi^v$ is invariant under the transformation $X\mapsto X^S$.

\subsection{On Jeulin's limit theorem}
\label{jeulin}

In this subsection, we shall give a proof of \eqref{ourversionjeulin}. This result is analogous to Jeulin's limit theorem for the normalized Brownian excursion but it's verification will not rely on the delicate estimates used by the aforementioned author in \cite{jeulinsmg}. This is because we stop our processes at fixed times instead of the random times of last visit. 

  We recall Jeulin's theorem, which was introduced and proved in \cite{jeulinsmg}: if $\e$ is a normalized Brownian excursion and  we define $X^\eps=\paren{X^\eps_t}_{t\leq r/\eps^2}$ and $Y^\eta=\paren{Y^\eta_t}_{t\leq (1-r)/\eta^2}$ by\begin{align*}
X^\eps_t=\frac{1}{\eps}\e_{\eps^2 t}
\quad\text{and}\quad
Y^\eta_t=\frac{1}{\eta}\e_{1-\eta^2 t},
\end{align*}then the law of $\paren{X^\eps,Y^\eta}$, both coordinates  stopped when last visiting $a\geq 0$ before times $r$ and $1-r$ respectively, converges in variation as $\paren{\eps,\eta}\to 0$ to the law of two independent Bessel processes of dimension three starting at zero and killed on their last visit to $a$. (The  formulation in\cite{jeulinsmg} does not mention convergence in variation; this is implied by the proof.)

Let us now discuss equation \eqref{ourversionjeulin}. We shall work on the canonical spaces $\mc{C}^v$ where the laws\begin{esn}
\set{\pi_{x,y}^v:x,y,v>0},\end{esn}
$\pi_{x,y}^v$ corresponding to a Brownian bridge from $x$ to $y$ of length $v$ conditioned on remaining positive,  are defined. We will denote by $X$ and $\paren{\F_t}_{t\geq 0}$ the canonical process and filtration. 

As $y\to 0$, $\pi^{v}_{x,y}$ has a weak limit which shall be denoted $\pi_y^v$; this law satisfies a local absolute continuity with respect to the law of the three-dimensional Bessel process starting at zero, denoted $\p_0^3$, of the following form: $\pi_y^v|_{\F_s}$ is absolutely continuous with respect to $\p_0^3|_{\F_s}$ and the Radon-Nikod\'ym derivative $D_y^{v,s}$ can be written in terms of the canonical process $X$, the transition density $q_s$ of Brownian motion killed when it reaches zero and the density $f_x$ of the hitting time of $x$ by a Brownian motion started at zero as follows:\begin{esn}
\imf{D_{y}^{v,s}}{X}=\frac{\imf{q_{v-s}}{X_s,y}}{2\imf{f_y}{v}X_s}. 
\end{esn}From this, one might infer an inhomogeneous Markov property for $\pi_{y}^v$.

We shall use the following facts:  $\pi^v$ is the weak limit of $\pi_y^v$ as $y\to 0$ which satisfies the following result, combining its inhomogeneous Markov property  with time-reversibility: if  $\tilde X$ is the time-reversed process given  by $\tilde{X}_s=X_{v-s}$ and $\Phi$ is functional on the two-fold product of canonical space with itself which is positive and $\F_{s_1}\otimes\F_{s_2}$\nbd measurable, then for $0<s_i<v-s_2<v$\begin{equation}
\label{bridgeandreversibility}
\imf{\pi^v}{\imf{\Phi}{X,\tilde X}}=\imf{\pi^v}{\imf{\pi^{s_1}_{X_{s_1}}\otimes \pi^{s_2}_{\tilde{X}_{s_2}}}{\Phi}}.
\end{equation}We shall use this to establish a preliminary version of Jeulin's theorem with deterministic times in lieu of random ones. Consider the scaling operator $\fun{\mc{S}_{u}}{\mc{C}^v}{\mc{C}^{v/u}}$ defined as follows as follows: $\imf{\mc{S}_vf}{s}=\imf{f}{us}/\sqrt{u}$. Then, for $0<s_1<1-s_2<1$, as $\paren{\eps,\eta}\to \paren{0,0}$:\begin{equation}
\label{jeulinfixedtime}
\sup_{\|\Phi\|_{\infty}\leq 1}\abs{\imf{\pi^1}{\imf{\Phi}{\mc{S}_{\eps^2}\circ X,\mc{S}_{\eta^2}\circ \tilde X}}-\imf{\p_0^3\otimes\p_0^3}{\Phi}}\to 0.
\end{equation}In other words, Jeulin's theorem holds if we stop the processes at a fixed times instead of the random times of last visit. To see that \eqref{jeulinfixedtime} holds,  we need only remark that, from the explicit expressions\begin{esn}
\imf{q_s}{x,y}=\frac{1}{\sqrt{2\pi s}}\paren{e^{-\paren{y-x}^2/2s}-e^{\paren{x+y}^2/2s}}
\quad\text{and}\quad
\imf{f_x}{s}=\frac{x}{\sqrt{2\pi s^3}}e^{-x^2/2s},
\end{esn}the convergence\begin{equation}
\label{convergenceofdensity}
\imf{D_{y\sqrt{v}}^{s\sqrt{v},s_1}}{X}\to 1
\end{equation}as $v\to\infty$ with the other arguments fixed follows. To use the preceding asymptotic equivalence, we shall work on the threefold product of canonical space  with itself, with a measure constructed from $\pi^1$ and $\p_0^3\otimes\p_0^3$, and denote by $X$, $Y$ and $Z$ the first, second and third coordinate processes; $X$ will be used when integrating against $\pi^1$. By the bridge property \eqref{bridgeandreversibility} used with times $s$ and $1-s$, the local absolute continuity between $\pi^v_y$ and $\p_0^3$ and the scaling property $\pi_y^s\circ \mc{S}_{\eps^2}=\pi_{y/\eps}^{s/\eps^2}$, we may write\begin{align*}
&\abs{\imf{\pi^1}{\imf{\Phi}{\mc{S}_{\eps^2}\circ X,\mc{S}_{\eta^2}\circ \tilde X}}-\imf{\p_0^3\otimes\p_0^3}{\Phi}}
\\&\leq \|\Phi\|_\infty  \imf{\pi^1}{\imf{\p_0^3\otimes \p_0^3}{\abs{1-\imf{D_{X_s/\eps}^{s/\eps^2,s_1}}{Y}\imf{D_{X_s/\eps}^{\paren{1-s}/\eps^2,s_2}}{Z}}}}
\end{align*}for small $\eps$ and $\eta$. Since $\pi^1$-almost surely $x\mapsto \imf{\p_0^3\otimes \p_0^3}{\imf{D_{x/\eps}^{s/\eps^2,s_1}}{Y}\imf{D_{x/\eps}^{\paren{1-s}/\eps^2,s_2}}{Z}}$  converges to $1$ by \eqref{convergenceofdensity}, and it integrates $1$, it follows that the convergence holds also in $L^1$, proving \eqref{jeulinfixedtime}, which is actually stronger than \eqref{ourversionjeulin}.
\subsection{A law of the iterated logarithm}
In this subsection, we will prove Theorem \ref{lil} using the notation and preliminaries of Subsection \ref{haas}. Note that $H_t\leq M_t$ and that if we define $L_t$ as the length of the smallest closed interval that contains $\hat F^2_t$, then $M_t\leq L_t$. Let\begin{esn}
\imf{f}{t}=\frac{2t^2}{\log\abs{\log t}}
\end{esn}

We shall prove that\begin{esn}
\imf{\pi^1}{\liminf_{t\to 0+}\frac{H_t}{\imf{f}{t}}=1=\liminf_{t\to 0+}\frac{L_t}{\imf{f}{t}}}=1
\end{esn}which implies Theorem \ref{lil}. Let $T_m$ be the first hitting-time of $\set{m}$ and $L_t$ the last visit to $\set{t}$ (beware that $L_t$ stands for two different things). Let us recall that thanks to Williams decomposition of the It\^o measure $n_+$ and his time reversibility result we have\begin{align}
\label{lilCond}&\imf{\pi^1}{\liminf_{t\to 0+}\frac{L_t}{\imf{f}{t}}=1}
\\&=\int \frac{dm}{2m^2}\, \imf{\p_0^3\otimes\p_0^3}{\cond{\liminf_{t\to 0+}\frac{\imf{L_t}{X}+\imf{L_t}{Y}}{\imf{f}{t}}=1}{\imf{T_m}{X}+\imf{T_m}{Y}=1}}.\nonumber
\end{align}Under $\p_0^3\otimes \p_0^3$, $\paren{\imf{L_t}{X}+\imf{L_t}{Y}}_{t\geq 0}$ is a stable subordinator with Laplace exponent $\lambda\mapsto 2\sqrt{2\lambda}$, and by the iterated logarithm law for subordinators (cf. \cite[Thm. 11, p. 88]{bertoinpl}) we have\begin{esn}
 \imf{\p_0^3\otimes\p_0^3}{\liminf_{t\to 0+}\frac{\imf{L_t}{X}+\imf{L_t}{Y}}{\imf{f}{t}}=1}=1
\end{esn}and so:\begin{esn}
\imf{\p_0^3\otimes\p_0^3}{\cond{\liminf_{t\to 0+}\frac{\imf{L_t}{X}+\imf{L_t}{Y}}{\imf{f}{t}}=1}{\imf{T_m}{X}+\imf{T_m}{Y}=1}}=1.
\end{esn}We can therefore conclude:\begin{esn}
\imf{\pi^1}{\liminf_{t\to 0+}\frac{L_t}{\imf{f}{t}}=1}=1;
\end{esn}since $H_t\leq L_t$, then $\liminf_{t\to 0+}H_t/ \imf{f}{t}\leq 1$ $\pi^1$-almost surely and so it remains to prove a lower bound for this quantity. To do so, note that\begin{align*}
&\imf{\pi^1}{\liminf_{t\to 0+}\frac{H_t}{\imf{f}{t}}\geq 1}
\\&=\int \frac{dm}{2m^2}\, \imf{\p_0^3\otimes\p_0^3}{\cond{\liminf_{t\to 0+}\frac{\imf{T_t}{X}+\imf{T_t}{Y}}{\imf{f}{t}}\geq 1}{\imf{T_m}{X}+\imf{T_m}{Y}=1}}.
\end{align*}We will prove the equality\begin{equation}
\label{lilForBesT}
\imf{\p_0^3\otimes\p_0^3}{\liminf_{t\to 0+}\frac{\imf{T_t}{X}+\imf{T_t}{Y}}{\imf{f}{t}}\geq 1}=1
\end{equation}from which we deduce
\begin{esn}
\imf{\p_0^3\otimes\p_0^3}{\cond{\liminf_{t\to 0+}\frac{\imf{T_t}{X}+\imf{T_t}{Y}}{\imf{f}{t}}\geq 1}{\imf{T_m}{X}+\imf{T_m}{Y}=1}}=1
\end{esn}so that the $\pi^1$-almost sure lower bound $\liminf_{t\to 0+}H_t /\imf{f}{t}\geq 1$ follows, proving Theorem \ref{lil}. 

It remains to prove \eqref{lilForBesT}. This can be done based on the simple Lemma 3.1 of \cite{watanabeIntegralTestsForClassL} which translates in our case as follows. 
\begin{lem}[Watanabe, \cite{watanabeIntegralTestsForClassL}]
\label{WatanabesLemma}
Let $F$ be the distribution function of $\imf{T_1}{X}+\imf{T_1}{Y}$ under $\p_0^3\otimes\p_0^3$. If for all $c<1$, the integral\begin{equation}
\label{integral}
\int_{0+}\frac{1}{t}\imf{F}{\frac{2c}{\log\abs\log t}}\, dt
\end{equation}is finite, then \eqref{lilForBesT} holds. 
\end{lem}To study the integral appearing in the preceding lemma, we use the following result:
\begin{lem}
\label{BertoinsBound}
For all $c>0$,\begin{esn}
\liminf_{t\downarrow 0}-\frac{\log\imf{F}{\frac{2c}{\log\abs\log t}}}{\log\abs\log t}\,\geq \frac{1}{c}.
\end{esn}
\end{lem}
Lemma \ref{BertoinsBound} implies that given $c<1$, there exists $\eta>1$ such that\begin{esn}
\imf{F}{\frac{2c}{\log\abs\log t}}\leq \frac{1}{\abs{\log t}^\eta}
\end{esn}for all small enough $t$, so that the integral  in \eqref{integral} is finite. Hence, it only remains to prove Lemma \ref{BertoinsBound}. This is done by following part of the reasoning used to prove Lemma 12 in \cite[III, p.88]{bertoinpl}; we will present only a sketch. 
\begin{proof}[Proof of Lemma \ref{BertoinsBound}]
Since\begin{esn}
\imf{\p_0^3}{e^{-\lambda T_1}}=\frac{\sqrt{2\lambda}}{\imf{\sinh}{\sqrt{2\lambda}}},
\end{esn}then by defining\begin{esn}
\int_0^\infty e^{-\lambda x}\,\imf{F}{dx}=e^{-\imf{\Phi}{\lambda}},
\end{esn}we have $\imf{\Phi}{\lambda}\sim 2\sqrt{2}\lambda$ is $\lambda\to\infty$. Let $\imf{g}{t}=2/\log\abs\log{t}$ and use Chevyshev's inequality to get\begin{equation}
\label{chebyshev}
-\log\imf{F}{c/\imf{g}{t}}\geq \imf{\Phi}{\lambda}-\lambda c \imf{g}{t}.
\end{equation}Let $\phi$ stand for the inverse of $\Phi$, so that $\imf{\phi}{\lambda}\sim \lambda^2/8$ as $\lambda\to\infty$ and take $\lambda=\imf{\phi}{\kappa \log\abs\log t}$ (for some constant $\kappa$ to be specified in a moment). Then the lower bound of \eqref{chebyshev} is asymptotic to\begin{esn}
\paren{\kappa-\kappa\frac{c}{4}}\log\abs\log t,
\end{esn}which attains its maximum $\log\abs\log t\,/c$ when $\kappa=2/c$.
\end{proof}
\section*{Acknowledgements}
I would like to express my gratitude towards my PhD supervisors  Jean Bertoin and Mar\ii a Emilia Caballero for their patience, guidance, and support.  Discussions with Benedicte Haas were crucial to the work at hand, especially regarding the path transformation which she was kind enough to let me present. Part of this work was carried out during while visiting the {\it Laboratoire de Probabilit\'es et Mod\`eles Al\'eatoires} of Paris VI University; the hospitality of its members is heartily thanked. 
Finally, I wish to thank the anonymous referee who made acute remarks on content and style and who suggested to look for a result on the lines of Theorem \ref{lil}.
\bibliography{browfrag}
\bibliographystyle{amsplain}
\end{document}